\documentclass[11pt,a4paper,reqno]{amsart}
\usepackage[T1]{fontenc}
\usepackage[utf8]{inputenc}
\usepackage{lmodern}
\usepackage[margin=28mm]{geometry}
\usepackage{amssymb,mathtools}
\usepackage{microtype}
\usepackage{enumitem}
\usepackage{needspace}
\usepackage{xcolor}
\usepackage{hyperref}

\AddToHook{cmd/normalsize/after}{%
  \setlength{\abovedisplayskip}{7pt plus 2pt minus 2pt}%
  \setlength{\belowdisplayskip}{7pt plus 2pt minus 2pt}%
  \setlength{\abovedisplayshortskip}{3pt plus 1pt minus 1pt}%
  \setlength{\belowdisplayshortskip}{7pt plus 2pt minus 2pt}%
}

\definecolor{linkburgundy}{HTML}{982B45}

\hypersetup{
  colorlinks=true,
  allcolors=linkburgundy,
  pdftitle={Bimodule duality and Artin's problem for skew field extensions},
  pdfauthor={Niels Lauritzen and Jesper Funch Thomsen},
  pdfsubject={Bimodule duality and applications to known Calabi--Yau examples},
  pdfkeywords={Artin's problem for skew field extensions, division algebra, left and right indices, bimodule duality, homological smoothness, twisted Calabi--Yau algebra}
}

\newcommand{\namedref}[3][]{\hyperref[#3]{#2~\ref*{#3}#1}}

\newtheorem{theorem}{Theorem}[section]
\newtheorem{proposition}[theorem]{Proposition}
\newtheorem{lemma}[theorem]{Lemma}
\newtheorem{corollary}[theorem]{Corollary}
\theoremstyle{definition}
\newtheorem{definition}[theorem]{Definition}
\newtheorem{example}[theorem]{Example}
\newtheorem{remark}[theorem]{Remark}

\numberwithin{equation}{section}
\DeclareMathOperator{\Hom}{Hom}
\DeclareMathOperator{\Ext}{Ext}
\DeclareMathOperator{\pd}{pd}

\newcommand{\op}{\mathrm{o}}
\newcommand{\FP}{\mathrm{FP}}
\newcommand{\ellind}{[D:E]_{\ell}}
\newcommand{\rind}{[D:E]_{r}}
\newcommand{\tensor}{\otimes}
\setlist[enumerate]{label=\textup{(\roman*)},leftmargin=*,itemsep=3pt,topsep=5pt}
\allowdisplaybreaks[1]
\title[Bimodule duality and Artin's problem]
{Bimodule duality and Artin's problem\\
for skew field extensions}

\author{Niels Lauritzen and Jesper Funch Thomsen}
\address{Department of Mathematics,
Aarhus University,
Ny Munkegade 118,
DK-8000 Aarhus C,
Denmark}
\email{niels@math.au.dk, funch@math.au.dk}
\subjclass[2020]{Primary 16K40; Secondary 16S32, 16E40}
\keywords{Artin's problem for skew field extensions, division algebras,
left and right indices, bimodule duality,
homological smoothness, twisted Calabi--Yau algebras}
\date{}

\begin{document}

\begin{abstract}
Emil Artin asked whether the left and right dimensions of an extension of
skew fields must agree.  Cohn and Schofield gave counterexamples through elaborate constructions. Call a
division ring index-symmetric if every division subring with at least one
finite index has equal finite indices. It has been a long-standing open
question whether there exists an index-symmetric division ring that is
infinite-dimensional over its center. We give the first such examples:
using a duality formula for bimodules over division algebras, we prove
that every twisted Calabi--Yau division algebra is index-symmetric. This
applies to the Weyl skew fields and, more generally, to the quotient
division rings of rings of differential operators on smooth affine
varieties in characteristic zero, to quantum skew fields, to the quotient
division rings of noetherian AS-regular Hopf domains, and to the quotient
and degree-zero division rings of three- and four-dimensional elliptic
Sklyanin algebras, which we show to be twisted Calabi--Yau. In particular,
every division ring on Artin's list of noncommutative surfaces is
index-symmetric.
\end{abstract}

\maketitle

\section{Introduction}
 
For an extension $E\subseteq D$ of division rings, put
$$
 [D:E]_{\ell}=\dim_E({}_E D),
 \qquad
 [D:E]_r=\dim_E(D_E).
$$

Emil Artin asked whether these dimensions must agree (see
\cite[p.~312]{Princeton} and
\cite[Chapter~I, \S8, p.~34]{ArtinGeometric}). The answer is negative
in general: Cohn \cite[Theorem~6 and p.~548]{Cohn} constructed
extensions that are finite on only one side, and
Schofield~\cite{Schofield} constructed extensions with any prescribed
pair of finite indices greater than one.

This raises the question of which division rings behave well. We call
a division ring $D$ \emph{index-symmetric} if every division subring
$E\subseteq D$ with at least one finite index satisfies
$[D:E]_\ell=[D:E]_r<\infty$. Division rings that are finite-dimensional
over their centers are index-symmetric
(\namedref{Corollary}{cor:centrally-finite-indices}), but beyond this case
little has been known. Schofield showed that for stratiform division
rings, a class that includes the Weyl skew fields and the quotient
division rings of enveloping algebras of solvable Lie algebras and of
group algebras of torsion-free polycyclic-by-finite groups, one finite index forces
the other to be finite \cite[Corollary~26]{Stratiform}, and he noted
that it is not known for any such division ring that is
infinite-dimensional over its center whether the two indices are equal
(see the remark after \cite[Corollary~26]{Stratiform}). The same
question was raised for the first Weyl skew field $D_1(\mathbb C)$ in
the report of the 2005 Banff workshop on noncommutative algebra and
algebraic geometry, where it is described as a long-standing open
problem \cite[p.~2]{Banff}.

We give the first examples of index-symmetric division rings that are
infinite-dimensional over their centers. In particular, we prove that
the Weyl skew fields in characteristic zero are index-symmetric,
answering the question above.

\subsection{Main results}

For a division $k$-algebra $D$ put
$D^e=D\tensor_k D^{\op}$ and $T_i(D)=\Ext^i_{D^e}(D,D^e)$. Recall that
$D$ is twisted Calabi--Yau of dimension $d$ if it is homologically
smooth, $T_i(D)=0$ for $i\ne d$ and $T_d(D)\cong D_\nu$ for an
automorphism $\nu$ of $D$.

Our main result is \namedref{Theorem}{thm:C}, which shows that every
twisted Calabi--Yau division algebra $D$ is index-symmetric. Moreover, a division
$k$-subalgebra $E\subseteq D$ with one finite index is itself twisted
Calabi--Yau of the same dimension, provided it has finite projective
dimension as a bimodule.

By \namedref{Proposition}{prop:reduction}, index-symmetry may be tested on
division $k$-subalgebras for a central subfield $k$, which makes
$k$-central bimodules and the enveloping algebra $D^e$ available.
\namedref{Theorem}{thm:C} is a special case of one-sided criteria in terms of the
dimensions of $T_i(D)$ as left and right $D$-vector spaces
(\namedref{Theorems}{thm:A} and~\ref{thm:B}).

\namedref{Section}{sec:natural-examples} contains the applications of our
results. We recall the known twisted Calabi--Yau results needed for these
applications, giving proofs or references as appropriate. Using
\namedref{Theorem}{thm:C} and the
centrally finite case (\namedref{Corollary}{cor:centrally-finite-indices}),
we establish index-symmetry for the following division rings:
\begin{itemize}
\item the quotient division rings of rings of differential operators
on smooth affine varieties in characteristic zero, in particular the
Weyl skew fields $D_n(k)$;
\item the quotient division rings of Ore extensions of twisted
Calabi--Yau noetherian domains, in particular of Ore extensions of
function fields in characteristic zero;
\item the quantum skew fields $D_{\mathbf q}$;
\item the quotient division rings of noetherian AS-regular Hopf domains,
including enveloping algebras $U(\mathfrak g)$ of finite-dimensional
Lie algebras, the quantized enveloping algebras $U_q(\mathfrak g)$ and
quantized coordinate rings $\mathcal O_q(G)$ in the parameter setting
specified in \namedref{Subsection}{sec:examples-Hopf}, and the group
algebras $k[\Gamma]$ of torsion-free polycyclic-by-finite groups;
\item the quotient division rings and degree-zero division rings of
the three- and four-dimensional elliptic Sklyanin algebras
(\namedref{Example}{ex:examples-sklyanin});
\item all division rings in Michael Artin's proposed birational
classification of noncommutative surfaces \cite{ArtinSurfaces}
(\namedref{Corollary}{thm:D}).
\end{itemize}
This answers Schofield's question for the Weyl skew fields and for
$Q(U(\mathfrak g))$, $\mathfrak g$ solvable, and $Q(k[\Gamma])$. It also answers the question of
Bell, Ingalls and Ramkumar about equality of the left and right
indices of embeddings between division rings on Artin's list
\cite[\S3]{BIR}.
For the first Weyl skew field we obtain a complete
answer: if $\operatorname{char}k=0$ and $E\subseteq D_1(k)$ is a
division $k$-subalgebra, then $[D_1(k):E]_\ell=[D_1(k):E]_r<\infty$ if
$E$ is noncommutative, and both indices are infinite otherwise
(\namedref{Corollary}{cor:banff}). Here finiteness is due to Schofield
\cite[Corollary~35]{Stratiform}. This answers the question on equality
of left and right indices for $\mathbb C$-algebra embeddings of
$D_1(\mathbb C)$ into itself, recorded in the 2005 workshop report
\cite[p.~2]{Banff}. 

The main tool is a duality formula for a $k$-central bimodule
${}_A M_B$ over division $k$-algebras $A$ and $B$. When $M$ is
finite-dimensional on both sides and the finiteness hypotheses of
\namedref{Lemma}{lem:mixed} hold, we obtain
$$
 T_i(B)\tensor_B M^\ell
 \cong
 \Ext^i_{A\tensor_k B^{\op}}(M,A\tensor_k B^{\op})
 \cong
 M^r\tensor_A T_i(A),
$$
where $M^\ell$ and $M^r$ are the left and right duals of $M$. It uses
finite-projective evaluation as in the proof of Van den Bergh's duality
theorem \cite{VdB}, but degree by degree and without assuming that Ext
is concentrated in one degree or invertible. Comparing dimensions and
iterating the dual for $M={}_D D_E$ gives equality of the indices; the
finiteness hypotheses on $E$ come from those on $D$ by faithfully flat
descent.

\section{Preliminaries}
\label{sec:prelim}

Throughout, $k$ is a field, all algebras and modules are unital, and all
bimodules are central over $k$. A ring is \emph{noetherian} if it is left
and right noetherian. For a $k$-algebra $A$, we write $A^\op$ for its  opposite algebra
and $A^e=A\tensor_k A^{\op}$ for its enveloping algebra.
For $k$-algebras $A$ and $B$, we write
${}_A\mathrm{Mod}_B$ for the category of $(A,B)$-bimodules central over
$k$, with bimodule homomorphisms as morphisms. We identify this category
with the category of left $A\tensor_k B^{\op}$-modules via
$(a\tensor b^{\op})m=amb$.
We call $A\in {}_A\mathrm{Mod}_A$ with its natural bimodule structure 
\emph{the diagonal bimodule} or the \emph{left diagonal $A^e$-module}.

For a $k$-algebra $A$ and $i\geq 0$ put
$$
T_i(A)=\Ext^i_{A^e}(A,A^e).
$$
Here $T_i(A)$ is a right $A^e$-module by right multiplication in
$A^e$. We regard it as an $A$-bimodule via
$$
  a\xi b=\xi(b\tensor a^{\op}).
$$
For a division $k$-algebra $A$ write
$$
  p_i(A)=\dim_A({}_A T_i(A))
  \qquad\text{and}\qquad q_i(A)=\dim_A(T_i(A)_A).
$$
Dimensions are cardinals unless explicitly assumed finite. For an inclusion
$E\subseteq D$ of division $k$-algebras, we abbreviate
$\ell=\ellind$ and $r=\rind$.

A left module is of type $\FP_n$, for $n\geq0$, if it has a projective
resolution whose terms in degrees $0,\ldots,n$ are finitely generated.
It is of type $\FP_\infty$ if it has a resolution by finitely generated
projectives in every degree. A statement that an algebra $A$ is of type
$\FP_n$ \emph{as a bimodule} refers to the diagonal left $A^e$-module $A$.
No bound on the length of the resolution is part of this condition.
If $A^e$ is left noetherian, then $A$ is of type $\FP_\infty$ as a
bimodule, since it is cyclic as a left $A^e$-module.

A module is \emph{perfect} if it has a finite resolution by finitely
generated projective modules. An algebra $A$ is \emph{homologically smooth}
if its diagonal left $A^e$-module is perfect. In the automorphism-twist
convention, $A$ is \emph{twisted Calabi--Yau of dimension $d$} if it is
homologically smooth and
$$
  T_n(A)\cong
  \begin{cases}
    0,&n\ne d,\\
    A_\sigma,&n=d,
  \end{cases}
$$
where $\sigma$ is a $k$-algebra automorphism and $A_\sigma$ has action
$a\cdot u\cdot b=au\sigma(b)$. When $\sigma$ is the identity, we say
that $A$ is \emph{Calabi--Yau of dimension $d$}.

\section{Index-symmetric division rings}
\label{sec:reduction}

We introduce the notion of index-symmetry and show that it suffices to
verify it for division $k$-subalgebras, where $k$ is a central subfield.

\begin{definition}
\label{def:index-symmetric}
A division ring $D$ is \emph{index-symmetric} if every division subring
$E\subseteq D$ with at least one finite index satisfies
$$
 [D:E]_\ell=[D:E]_r<\infty.
$$
If $k\subseteq Z(D)$ is a subfield, $D$ is \emph{index-symmetric over
$k$} if this holds for every division $k$-subalgebra $E\subseteq D$.
\end{definition}

Equivalently, $D$ is index-symmetric if, for every division subring
$E\subseteq D$, either both indices are infinite or they are finite and
equal. The definition says nothing about whether two infinite indices
are equal as cardinals.

Clearly, an index-symmetric division ring is index-symmetric over every
central subfield $k$. \namedref{Proposition}{prop:reduction} below shows that
the converse also holds: if $D$ is index-symmetric over some central
subfield $k$, then $D$ is index-symmetric.

\begin{lemma}
\label{lem:Ek}
Let $R$ be a domain, $E\subseteq R$ a division subring, and $k\subseteq R$
a subfield whose elements commute with $E$. Put
$Ek=\{\sum_i e_ic_i : e_i\in E,\ c_i\in k\}$.
\begin{enumerate}
\item $Ek$ is a subring of $R$, and $\dim_E({}_E(Ek))=\dim_E((Ek)_E)$.
\item If this dimension is finite, then $Ek$ is a division ring.
\end{enumerate}
\end{lemma}

\begin{proof}
(i) Since $k$ is commutative and commutes with $E$, $Ek$ is closed under
multiplication, and $\sum_i e_ic_i=\sum_i c_ie_i$ for $e_i\in E$,
$c_i\in k$. Hence a subset of $k$ is left $E$-linearly independent if and
only if it is right $E$-linearly independent. As $k$ spans $Ek$ on both
sides, a left $E$-basis of $Ek$ chosen inside $k$ is also a right
$E$-basis.

(ii) Let $0\ne s\in Ek$. Right multiplication by $s$ is an injective left
$E$-linear endomorphism of the finite-dimensional space ${}_E(Ek)$, hence
surjective, so $ys=1$ for some $y\in Ek$. Then $y\ne0$, so $zy=1$ for some
$z\in Ek$, and $z=zys=s$. Thus $sy=ys=1$.
\end{proof}

\begin{proposition}
\label{prop:reduction}
Let $D$ be a division ring and let $k\subseteq Z(D)$ be a subfield.
Then $D$ is index-symmetric if and only if it is index-symmetric over $k$.
\end{proposition}

\begin{proof}
If $D$ is index-symmetric, it is clearly index-symmetric over $k$.
Conversely, suppose $D$ is index-symmetric over $k$, and let
$E\subseteq D$ be a division subring with $[D:E]_\ell<\infty$; the
case $[D:E]_r<\infty$ is symmetric. Since $k$ is central in $D$,
\namedref{Lemma}{lem:Ek} applies to $E\subseteq D$ and $k$. As
${}_E(Ek)\subseteq{}_ED$, the ring $Ek$ is finite-dimensional over $E$ on
the left, hence a division ring. Thus $E'=Ek$ is a division
$k$-subalgebra of $D$ with
$$
 [E':E]_\ell=[E':E]_r<\infty.
$$
By the tower formula, $[D:E']_\ell\le[D:E]_\ell<\infty$, so
index-symmetry over $k$ gives $[D:E']_\ell=[D:E']_r<\infty$. Hence
$$
 [D:E]_r=[D:E']_r\,[E':E]_r=[D:E']_\ell\,[E':E]_\ell=[D:E]_\ell<\infty.
 \eqno\hbox{\qedhere}
$$
\end{proof}

Consequently, to prove that $D$ is index-symmetric it suffices to
consider division subrings $E\subseteq D$ containing $k$, that is,
division $k$-subalgebras, and we do so throughout.

Taking $k=Z(D)$ settles the case of division rings that are finite over
their centers.

\begin{corollary}
\label{cor:centrally-finite-indices}
If $D$ is finite-dimensional over its center $Z$, then $D$ is
index-symmetric.
\end{corollary}

\begin{proof}
For every division $Z$-subalgebra $L\subseteq D$,
$[D:L]_\ell=[D:L]_r=\dim_ZD/\dim_ZL<\infty$. Thus $D$ is
index-symmetric over $Z$, and \namedref{Proposition}{prop:reduction} applies.
\end{proof}

\begin{remark}
For an extension $E\subseteq D$ of division rings, every left $E$-basis of
$D$ is a right $E$-basis if and only if $E=D$ or $E\subseteq Z(D)$. The
``if'' direction is clear. Conversely, suppose every left $E$-basis is a
right $E$-basis. For $x\in D^\times$ and $e\in E$, the elements $x$ and $xe$
are right dependent, hence by assumption left dependent, so $xe\in Ex$. Thus $xEx^{-1}\subseteq E$ for all $x\in D^\times$, and the
Cartan--Brauer--Hua theorem gives $E=D$ or $E\subseteq Z(D)$.
\end{remark}

\section{Duality for a bimodule between two division algebras}
\label{sec:mixed}

Let $A,B$ be division $k$-algebras and let $M\in{}_A\mathrm{Mod}_B$.
Set $C=A\tensor_k B^{\op}$, so that $M$ is a left $C$-module.
We regard a right $C$-module $V$ as a $B$-$A$-bimodule via
$$
  bva=v(a\tensor b^{\op}).
$$
Define the two duals
$$
  M^\ell=\Hom_A({}_A M,{}_A A),
  \qquad M^r=\Hom_{B^{\op}}(M_B,B_B).
$$
Both duals belong to ${}_B\mathrm{Mod}_A$, with actions
\begin{equation}
  (bfa)(x)=f(xb)a,
  \qquad (bga)(x)=b\,g(ax).
  \label{eq:dual-actions}
\end{equation}

If $\alpha=\dim_A({}_A M)<\infty$ and $\beta=\dim_B(M_B)<\infty$, then
\begin{equation}
  \dim_A(M^\ell_A)=\alpha,
  \qquad \dim_B({}_B M^r)=\beta.
  \label{eq:known-dual-dimensions}
\end{equation}
Notice that finiteness on the other sides of these duals is not automatic.

\begin{lemma}
\label{lem:one-sided-FP}
Let $M\in{}_A\mathrm{Mod}_B$.
\begin{enumerate}
\item If $\dim_A({}_A M)<\infty$ and $Q_\bullet\to B$ is a projective
$B^e$-resolution, then $M\tensor_B Q_\bullet\to M$ is a projective
$C$-resolution, finitely generated in every degree where $Q_\bullet$ is.
\item If $\dim_B(M_B)<\infty$ and $P_\bullet\to A$ is a projective
$A^e$-resolution, then $P_\bullet\tensor_A M\to M$ is a projective
$C$-resolution, finitely generated in every degree where $P_\bullet$ is.
\end{enumerate}
In particular, $M$ is of type $\FP_n$ over $C$ if ${}_AM$ is
finite-dimensional and $B$ is of type $\FP_n$ as a bimodule, or if $M_B$
is finite-dimensional and $A$ is of type $\FP_n$ as a bimodule.
\end{lemma}

\begin{proof}
For (i), $M$ is flat as a right $B$-module, so $M\tensor_B Q_\bullet$
resolves $M\tensor_B B\cong M$. Moreover, $M\tensor_B B^e\cong
M\tensor_k B^{\op}$ is a free left $C$-module of rank $\dim_A({}_AM)$,
so $M\tensor_B-$ takes (finitely generated) projective $B^e$-modules to
(finitely generated) projective $C$-modules. Part (ii) is symmetric.
\end{proof}

\Needspace{10\baselineskip}
\begin{lemma}
\label{lem:mixed}
Fix $i\geq0$.
\begin{enumerate}
\item If $B$ is of type $\FP_{i+1}$ as a bimodule and ${}_A M$ is
finite-dimensional, there is an isomorphism in ${}_B\mathrm{Mod}_A$
$$
  \Ext^i_C(M,C)\cong T_i(B)\tensor_B M^\ell.
$$
\item If $A$ is of type $\FP_{i+1}$ as a bimodule and $M_B$ is
finite-dimensional, there is an isomorphism in ${}_B\mathrm{Mod}_A$
$$
  \Ext^i_C(M,C)\cong M^r\tensor_A T_i(A).
$$
\end{enumerate}
Consequently, when both hypotheses hold, there are isomorphisms in
${}_B\mathrm{Mod}_A$
\begin{equation}
  T_i(B)\tensor_B M^\ell
  \cong \Ext^i_C(M,C)
  \cong M^r\tensor_A T_i(A).
  \label{eq:mixed}
\end{equation}
All tensor products have the outer bimodule actions.
\end{lemma}

\begin{proof}
For (i), take a projective $B^e$-resolution $Q_\bullet\to B$ with finitely
generated terms through degree $i+1$. By \namedref[(i)]{Lemma}{lem:one-sided-FP},
$M\tensor_B Q_\bullet$ is a projective $C$-resolution of $M$ with
finitely generated terms through degree $i+1$.

Put $H=\Hom_A({}_A M,{}_A C)$, where $A$ acts on $C$ through its first
factor. Then $H\in{}_{B^e}\mathrm{Mod}_C$, with actions
$$
  ((b\tensor c^{\op})h)(x)=(1\tensor c^{\op})h(xb),
  \qquad (h\gamma)(x)=h(x)\gamma.
$$
Regard $B^e$ as a right $B$-module through $b\mapsto b\tensor1$.
There is an isomorphism in ${}_{B^e}\mathrm{Mod}_C$
\begin{equation}
\begin{split}
  B^e\tensor_B M^\ell&\longrightarrow H,\\
  (b\tensor c^{\op})\tensor f
    &\longmapsto\bigl(x\longmapsto f(xb)\tensor c^{\op}\bigr),
\end{split}
\label{eq:H}
\end{equation}
where the right $C$-action on the source is
$$
  (\rho\tensor f)(a\tensor c^{\op})
    =\rho(1\tensor c^{\op})\tensor fa.
$$
Balancing and compatibility with the actions follow directly from
\eqref{eq:dual-actions}. Bijectivity follows from
$$
 B^e\tensor_B M^\ell\cong M^\ell\tensor_k B^{\op}
 \cong\Hom_A(M,A\tensor_k B^{\op});
$$
the second isomorphism uses the finite left dimension of $M$.

For a finitely generated projective left $B^e$-module $Q$, adjunction,
finite-projective evaluation, and \eqref{eq:H} give natural isomorphisms
of right $C$-modules
\begin{align*}
 \Hom_C(M\tensor_B Q,C)
 &\cong\Hom_{B^e}(Q,H)\\
 &\cong\Hom_{B^e}(Q,B^e)\tensor_{B^e}H\\
 &\cong\Hom_{B^e}(Q,B^e)\tensor_B M^\ell.
\end{align*}
The evaluation isomorphism used here holds for finite free modules and
hence for their direct summands. In the last expression the right
$C$-action is
$$
 (u\tensor f)(a\tensor c^{\op})
   =u(1\tensor c^{\op})\tensor fa.
$$
Thus the corresponding left $B$- and right $A$-actions are the outer
actions. Apply this identification to the part of $Q_\bullet$ in degrees $i-1,i,i+1$
(omitting degree $-1$ when $i=0$). Tensoring over $B$ is exact and hence
commutes with cohomology. This proves (i).

For (ii), apply (i) to ${}_{B^{\op}}M^{\op}_{A^{\op}}$, where
$b^{\op}\cdot x\cdot a^{\op}=axb$, and then reverse the sides.
The algebra isomorphism
$B^{\op}\tensor_k A\longrightarrow A\tensor_k B^{\op}$,
$b^{\op}\tensor a\longmapsto a\tensor b^{\op}$, identifies the
corresponding left modules and their resolutions. The right actions on
Ext become the reversed bimodule actions. In particular,
$(M^{\op})^\ell\cong(M^r)^{\op}$ and
$T_i(A^{\op})\cong T_i(A)^{\op}$, and reversal changes the order of the
tensor factors. This gives precisely the displayed formula in (ii).
\end{proof}

\section{Iteration of duals and preservation of dimensions}
\label{sec:iteration}

For division $k$-algebras $A,B$ and $M\in{}_A\mathrm{Mod}_B$, the
\emph{dimension pair} of $M$ over $(A,B)$ is the ordered pair
$$
  \bigl(\dim_A({}_A M),\,\dim_B(M_B)\bigr).
$$

\begin{proposition}
\label{prop:iteration}
Let $A$ and $B$ be division $k$-algebras of type $\FP_{i+1}$ as bimodules,
and let $0\ne M\in{}_A\mathrm{Mod}_B$ be finite-dimensional on both sides. Suppose
$0<p_i(A)<\infty$. Then
$$
  p_i(B)=p_i(A),
$$
and $M^\ell$ is finite-dimensional on both sides, with
$$
  \dim_B({}_B M^\ell)=\dim_B(M_B),
  \qquad
  \dim_A(M^\ell_A)=\dim_A({}_A M).
$$
In particular, taking the left dual interchanges the two entries of the
dimension pair of $M$.
\end{proposition}

\begin{proof}
Write $\alpha=\dim_A({}_A M)$, $\beta=\dim_B(M_B)$ and $p_A=p_i(A)$.
By \eqref{eq:known-dual-dimensions}, the left $B$-dimension of
$M^r\tensor_A T_i(A)$ is $\beta p_A$, a positive integer. If
$$
  p_B=p_i(B),\qquad s=\dim_B({}_B M^\ell),
$$
then the left $B$-dimension of the other side of \eqref{eq:mixed} is
$p_Bs$: choose a left $B$-basis of $M^\ell$ and tensor with $T_i(B)$.
Thus, as cardinals,
\begin{equation}
  p_Bs=\beta p_A.
  \label{eq:dimension-recurrence}
\end{equation}
Since $M^\ell\ne0$ and the right side is finite and nonzero, both $p_B$
and $s$ are positive integers. Therefore $M^\ell$ is finite on both
sides, and its dimension pair, now over $(B,A)$, is
$$
  \left(\beta\frac{p_A}{p_B},\,\alpha\right).
$$

Set $M^{(0)}=M$ and $M^{(m+1)}=(M^{(m)})^\ell$. Then
$$
  M^{(2m)}\in{}_A\mathrm{Mod}_B,
  \qquad M^{(2m+1)}\in{}_B\mathrm{Mod}_A
  \qquad(m\geq0).
$$
We have just shown that
the next left dual is nonzero and finite on both sides. The same argument,
with $A$ and $B$ alternating, therefore applies indefinitely. After $2m$
steps the dimension pair over $(A,B)$ is
\begin{equation}
  \left(
    \alpha\left(\frac{p_B}{p_A}\right)^m,
    \beta\left(\frac{p_A}{p_B}\right)^m
  \right).
  \label{eq:iterated-dimensions}
\end{equation}
Every entry is a positive integer for every $m\geq0$. If $p_A\ne p_B$,
one of the entries in \eqref{eq:iterated-dimensions} tends to zero as
$m\to\infty$, a contradiction. Hence $p_A=p_B$, and
\eqref{eq:dimension-recurrence} gives $s=\beta$.
\end{proof}

\begin{remark}
This argument does not identify $M$ with its twice-iterated left dual.
It uses only the finiteness of every successive dual and the dimension
recurrence. A right-sided version follows by passing to opposite algebras.
\end{remark}

\section{Descent of finiteness for resolutions}
\label{sec:FP}

The resolutions required for the smaller division algebra will be obtained
by faithfully flat descent. We include the argument over arbitrary rings.

\begin{lemma}
\label{lem:flat-FP}
Let $S\to R$ be a unital ring homomorphism such that $R$ is faithfully flat
as a right $S$-module. For a left $S$-module $N$ and $n\geq0$, $N$ is
of type $\FP_n$ if and only if $R\tensor_S N$ is of type $\FP_n$ over
$R$. The same holds for $\FP_\infty$.
\end{lemma}

\begin{proof}
We use \cite[VIII, Proposition~4.3(iii)]{Brown} in the following form:
if $M$ is of type $\FP_n$ and $F\to M$ is a surjection from a finitely
generated free module, then it extends to an exact sequence
$$
 F_n\longrightarrow\cdots\longrightarrow F_1\longrightarrow F
 \longrightarrow M\longrightarrow0
$$
with $F_1,\ldots,F_n$ finitely generated free. In particular, for
$n\geq1$ the kernel of $F\to M$ is of type $\FP_{n-1}$. Conversely, if
the kernel of $F\to M$ is of type $\FP_{n-1}$, then $M$ is of type
$\FP_n$, by splicing a resolution of the kernel with $F\to M$.

We now prove the equivalence in the lemma. If $N$ is of type $\FP_n$,
then so is $R\tensor_S N$, by tensoring a resolution of $N$ with the
flat module $R$. We prove the converse by induction on $n$. For $n=0$,
write finitely many generators of $R\tensor_S N$ as finite sums of pure
tensors, and let $N_0$ be generated by their second factors. Then
$R\tensor_S(N/N_0)=0$, so faithful flatness gives $N=N_0$. For $n\geq1$, $N$ is finitely generated by the case $n=0$, so there is an
exact sequence
$$
 0\longrightarrow K\longrightarrow S^t\longrightarrow N\longrightarrow0,
 \qquad t<\infty.
$$
By flatness, the sequence
$$
 0\longrightarrow R\tensor_S K\longrightarrow R^t\longrightarrow
 R\tensor_S N\longrightarrow0
$$
is exact, so $R\tensor_S K$ is of type $\FP_{n-1}$ over $R$ by the
first part of the proof. By induction, $K$ is of type $\FP_{n-1}$ over
$S$, and hence $N$ is of type $\FP_n$, again by the first part of the
proof.

Finally, a module is of type $\FP_\infty$ if and only if it is of type
$\FP_n$ for every $n$ \cite[VIII, Proposition~4.5]{Brown}, so the
statement for $\FP_\infty$ follows from the statement for $\FP_n$.
\end{proof}

\begin{lemma}
\label{lem:index-FP}
Let $E\subseteq D$ be division $k$-algebras. If at least one of
$\ellind,\rind$ is finite and $D$ is of type $\FP_n$ as a bimodule,
then $E$ is of type $\FP_n$ as a bimodule. The same holds for
$\FP_\infty$.
\end{lemma}

\begin{proof}
Put $N=D\tensor_E D$, with the outer $D$-actions. Then
$D^e\tensor_{E^e}E\cong N$ as left $D^e$-modules. Since $D$ is free on
both sides over $E$, so is $D^e$ over $E^e$; in particular, $D^e$ is
faithfully flat as a right $E^e$-module. Moreover, $\dim_D({}_D N)=\ell$
and $\dim_D(N_D)=r$, so $N$ is finite-dimensional on at least one side.
Hence $N$ is of type $\FP_n$ over $D^e$ by \namedref{Lemma}{lem:one-sided-FP},
and \namedref{Lemma}{lem:flat-FP} shows that $E$ is of type $\FP_n$ over $E^e$.
The same argument applies to $\FP_\infty$.
\end{proof}

\section{Indices and Ext dimensions}
\label{sec:indices}

\begin{theorem}
\label{thm:A}
Let $E\subseteq D$ be division $k$-algebras, and let $i\geq0$.
Assume that $D$ is of type $\FP_{i+1}$ as a bimodule.
\begin{enumerate}
\item If $0<p_i(D)<\infty$ and $\rind<\infty$, then
$$
  \ellind=\rind<\infty.
$$
\item If $0<q_i(D)<\infty$ and $\ellind<\infty$, then the same conclusion holds.
\end{enumerate}
\end{theorem}

\begin{proof}
Suppose first that $0<p_i(D)<\infty$ and $r<\infty$.
\namedref{Lemma}{lem:index-FP} shows that $E$ is of type $\FP_{i+1}$ as a
bimodule. Apply \namedref{Proposition}{prop:iteration} with
$$
 A=D,\qquad B=E,\qquad M={}_D D_E.
$$
The dimension pair of $M$ is $(1,r)$. Evaluation at $1$ identifies
$$
 M^\ell=\Hom_D({}_D D,{}_D D)\cong{}_E D_D:
 \quad f\longmapsto f(1).
$$
Indeed, $f(x)=xf(1)$, and the actions in \eqref{eq:dual-actions} become
the usual left $E$- and right $D$-actions. \namedref{Proposition}{prop:iteration}
therefore gives $\dim_E({}_E D)=r$, proving (i).

For (ii), pass to $E^{\op}\subseteq D^{\op}$. The bimodule $\FP$ condition
is preserved under reversal, $p_i(D^{\op})=q_i(D)$, and the right index
of this opposite extension is $\ellind$. Part (i) applies.
\end{proof}

The same mixed-bimodule duality also compares the Ext dimensions directly.

\Needspace{8\baselineskip}
\begin{proposition}
\label{prop:comparison}
Let $E\subseteq D$ be division $k$-algebras and let $n\geq0$.
Suppose that $D$ is of type $\FP_{n+1}$ as a bimodule.
\begin{enumerate}
\item If $\ell<\infty$, then $q_n(D)\ell=q_n(E)r$ as cardinals.
\item If $r<\infty$, then $p_n(D)r=p_n(E)\ell$ as cardinals.
\end{enumerate}
In either case, $T_n(D)=0$ if and only if $T_n(E)=0$.
\end{proposition}

\begin{proof}
We prove (ii). \namedref{Lemma}{lem:index-FP} supplies the required
$\FP_{n+1}$ condition on $E$. Apply \namedref{Lemma}{lem:mixed} to
$M={}_D D_E$, which has dimension pair $(1,r)$. Since
$M^\ell\cong{}_E D_D$, it gives an isomorphism in ${}_E\mathrm{Mod}_D$
\begin{equation}
 T_n(E)\tensor_E D
 \cong
 \Hom_{E^{\op}}(D_E,E_E)\tensor_D T_n(D).
 \label{eq:direct-comparison}
\end{equation}
A left $E$-basis of $D$ shows that the left $E$-dimension of the first
side is $p_n(E)\ell$. On the other side,
$\Hom_{E^{\op}}(D_E,E_E)$ has left $E$-dimension $r$; choosing a left
$D$-basis of $T_n(D)$ therefore gives dimension $rp_n(D)$. These
calculations hold for cardinals, and prove (ii).

Apply (ii) to $E^{\op}\subseteq D^{\op}$ to obtain (i).
Since both indices are nonzero, either identity gives the asserted
equivalence of vanishing.
\end{proof}

\begin{theorem}
\label{thm:B}
Let $D$ be a division $k$-algebra of type $\FP_\infty$ as a bimodule.
Suppose that, for some $i,j\geq0$,
$$
  0<p_i(D)<\infty,
  \qquad 0<q_j(D)<\infty.
$$
If $E\subseteq D$ is a division $k$-subalgebra with finite index on at least
one side, then $E$ is of type $\FP_\infty$ as a bimodule,
$$
  \ellind=\rind<\infty,
$$
and, for every $n\geq0$,
\begin{equation}
  p_n(E)=p_n(D),
  \qquad q_n(E)=q_n(D).
  \label{eq:all-dimensions}
\end{equation}
The equalities hold also for infinite cardinals.
\end{theorem}

\begin{proof}
If $r<\infty$, apply \namedref[(i)]{Theorem}{thm:A} in degree $i$; if
$\ell<\infty$, apply \namedref[(ii)]{Theorem}{thm:A} in degree $j$.
In either case, $\ell=r<\infty$.
\namedref{Lemma}{lem:index-FP} gives type $\FP_\infty$ for $E$.
\namedref{Proposition}{prop:comparison} applies in every degree and gives
$$
 p_n(D)\ell=p_n(E)\ell,
 \qquad q_n(D)\ell=q_n(E)\ell.
$$
Multiplication by a nonzero finite cardinal is injective on cardinals:
it is injective on finite cardinals and fixes every infinite cardinal.
Consequently the common factor $\ell$ cancels, giving
\eqref{eq:all-dimensions}.
\end{proof}

\begin{theorem}
\label{thm:C}
Let $D$ be a division $k$-algebra that is twisted Calabi--Yau of dimension
$d$. Let $E\subseteq D$ be a division $k$-subalgebra with finite index on
at least one side. Then $E$ is of type $\FP_\infty$ as a bimodule,
$$
  \ellind=\rind<\infty,
$$
and there is a $k$-algebra automorphism $\tau$ of $E$ such that
\begin{equation}
  T_n(E)\cong
  \begin{cases}
    0,&n\ne d,\\
    E_\tau,&n=d,
  \end{cases}
  \label{eq:CY-Ext}
\end{equation}
in ${}_E\mathrm{Mod}_E$. If, in addition, $\pd_{E^e}E<\infty$, then $E$ is
twisted Calabi--Yau of dimension $d$. In particular, $D$ is index-symmetric.
\end{theorem}

\begin{proof}
Since $D$ is homologically smooth, it is of type $\FP_\infty$; moreover
$p_d(D)=q_d(D)=1$ and all other $p_n(D),q_n(D)$ vanish.
\namedref{Theorem}{thm:B} gives equality of the indices, type $\FP_\infty$
for $E$, and the same Ext dimensions for $E$. Equality of the indices for
all such $E$ shows, by \namedref{Proposition}{prop:reduction}, that $D$ is
index-symmetric.

Since $T_d(E)$ is one-dimensional on both sides, $T_d(E)=Et=tE$ for any
$0\ne t\in T_d(E)$, and $tb=\tau(b)t$ defines a $k$-algebra automorphism
$\tau$ of $E$ with $E_\tau\cong T_d(E)$, $a\mapsto at$. This proves
\eqref{eq:CY-Ext}.

Finally, suppose $m=\pd_{E^e}E<\infty$. If $m=0$, then $E$ is
finitely generated projective over $E^e$. If $m>0$, take a resolution
$P_\bullet\to E$ by finitely generated projectives and let $K_m$ be
the image of $P_m\to P_{m-1}$. This image is finitely generated, and
dimension shifting gives
$$
 \Ext^1_{E^e}(K_m,V)\cong\Ext^{m+1}_{E^e}(E,V)=0
$$
for every left $E^e$-module $V$. A free surjection onto $K_m$ therefore
splits, so $K_m$ is projective. Truncating the resolution at $K_m$
shows that $E$ is homologically smooth. Together with
\eqref{eq:CY-Ext}, this proves that $E$ is twisted Calabi--Yau of
dimension $d$.
\end{proof}

The hypothesis $\pd_{E^e}E<\infty$ holds when $E$ is a direct summand
of $D$ in ${}_E\mathrm{Mod}_E$.
The following proposition is used in the appendix to establish the
twisted Calabi--Yau property for degree-zero division rings, and its
fixed-ring corollary gives a second argument for the Sklyanin division
rings on Artin's list (\namedref{Remark}{rem:type2-fixed-ring}).

\begin{proposition}
\label{prop:retraction}
Let $D$ be a division $k$-algebra that is twisted Calabi--Yau of
dimension $d$, and let $E\subseteq D$ be a division $k$-subalgebra with
finite index on at least one side. If $E$ is a direct summand of $D$ in
${}_E\mathrm{Mod}_E$, then $E$ is twisted Calabi--Yau of dimension $d$.
\end{proposition}

\begin{proof}
Since $D$ is homologically smooth, it has a resolution $Q_\bullet\to D$
by projective $D^e$-modules of finite length $m$. Products of a left and a
right $E$-basis of $D$ show that $D^e$ is free as a left $E^e$-module, so
$Q_\bullet$ is also a projective $E^e$-resolution of $D$, and
$\pd_{E^e}D\le m$. As $E$ is a direct summand of $D$ in
${}_E\mathrm{Mod}_E$, we get $\pd_{E^e}E\le\pd_{E^e}D<\infty$. The last
assertion of \namedref{Theorem}{thm:C}, which uses the finite index, shows
that $E$ is twisted Calabi--Yau of dimension $d$.
\end{proof}

\begin{corollary}
\label{cor:fixed-rings}
Let $D$ be a division $k$-algebra that is twisted Calabi--Yau of
dimension $d$, and let $G$ be a finite group of $k$-algebra automorphisms
of $D$ with $|G|$ invertible in $k$. Then $D^G$ is twisted Calabi--Yau
of dimension $d$.
\end{corollary}

\begin{proof}
The map $\pi(a)=|G|^{-1}\sum_{g\in G}g(a)$ is a morphism in
${}_{D^G}\mathrm{Mod}_{D^G}$ with $\pi|_{D^G}=\mathrm{id}$, so $D^G$ is a
direct summand of $D$ in ${}_{D^G}\mathrm{Mod}_{D^G}$. Artin's argument shows
$[D:D^G]_r\le|G|$: if $x_1,\ldots,x_{n}\in D$ with $n>|G|$, choose a
nonzero solution $(c_i)$ of $\sum_i g(x_i)c_i=0$ for all $g\in G$ with
the fewest nonzero entries and some $c_{i_0}=1$; applying $h\in G$ and
subtracting shows $h(c_i)=c_i$, so the $x_i$ are right linearly dependent
over $D^G$.
\namedref{Proposition}{prop:retraction} applies.
\end{proof}

\section{Applications to twisted Calabi--Yau algebras}
\label{sec:natural-examples}

We apply \namedref{Theorem}{thm:C} using the known dualities of rings of
differential operators, Weyl algebras, Ore extensions, quantum affine
spaces, and AS-regular Hopf algebras. The Hopf examples include ordinary
and quantized enveloping algebras, quantized coordinate rings, and the
group algebras considered below. We also include the Sklyanin examples;
the descent argument for their degree-zero division rings is given in
\namedref{Appendix}{sec:degree-zero-appendix}.
Throughout this section $k$ is an arbitrary field; characteristic zero is
assumed only where stated. 

\subsection{Passing to the quotient division ring}

The following is well known (cf.\ \cite[Theorem~4.7]{Baki});
we include the short argument for completeness.

\begin{lemma}
\label{lem:examples-localization}
Let $A$ be a left and right Ore $k$-domain with quotient division ring
$D=Q(A)$. If $A$ is homologically smooth, then so is $D$, and
\begin{equation}
 T_i(D)\cong D\tensor_A T_i(A)\tensor_A D
 \qquad(i\geq0)
 \label{eq:examples-localization}
\end{equation}
in ${}_D\mathrm{Mod}_D$. In particular, if $A$ is twisted Calabi--Yau of
dimension $d$, with $T_d(A)\cong A_\sigma$, then $D$ is twisted
Calabi--Yau of dimension $d$ and
$$
 T_d(D)\cong D_{\widetilde\sigma},
$$
where $\widetilde\sigma$ is the extension of $\sigma$ to $D$.
\end{lemma}

\begin{proof}
Choose a finite resolution $P_\bullet\to A$ by finitely generated
projective $A^e$-modules. Ore localization is flat on both sides, and
$$
 D^e\tensor_{A^e}A\cong D\tensor_A D\cong D,
$$
the first isomorphism given by $(a\tensor b^{\op})\tensor x\mapsto ax\tensor b$
and the second by multiplication. Thus
$D^e\tensor_{A^e}P_\bullet\to D$ is a finite projective resolution
with finitely generated terms. Adjunction and finite projectivity give
an isomorphism of complexes of right $D^e$-modules
$$
 \Hom_{D^e}(D^e\tensor_{A^e}P_\bullet,D^e)
 \cong \Hom_{A^e}(P_\bullet,A^e)\tensor_{A^e}D^e.
$$
Since $D^e$ is flat as a left $A^e$-module, tensoring commutes with
cohomology. For a right $A^e$-module $U$, our bimodule convention gives
$$
 U\tensor_{A^e}D^e\xrightarrow{\ \sim\ }
 D\tensor_A U\tensor_A D,
 \qquad u\tensor(a\tensor b^{\op})\longmapsto b\tensor u\tensor a.
$$
This proves \eqref{eq:examples-localization}. Every automorphism of $A$
preserves its nonzero elements and hence extends to $D$. Finally,
$$
 D\tensor_A A_\sigma\tensor_A D\xrightarrow{\ \sim\ }
 D_{\widetilde\sigma},
 \qquad d_1\tensor a\tensor d_2\longmapsto
 d_1a\widetilde\sigma(d_2),
$$
which proves the last assertion.
\end{proof}

\subsection{Differential operators and Weyl skew fields}

\begin{example}
\label{ex:examples-differential}
\label{ex:examples-weyl}
Let $k$ be a field of characteristic zero and let
$X=\operatorname{Spec}R$ be a smooth affine integral $k$-variety of
dimension $n$. The ring $\mathcal D_k(R)$ of $k$-linear differential
operators on $R$ is the subalgebra of $\operatorname{End}_k(R)$
generated by multiplication operators and $k$-derivations.

This is a noetherian domain which is homologically smooth
\cite[Theorem~1]{LLM} and Calabi--Yau of dimension $2n$
\cite[Corollary~2.7]{Yekutieli}. By \namedref{Lemma}{lem:examples-localization},
its quotient division ring $Q(\mathcal D_k(R))$ is also Calabi--Yau
of dimension $2n$, so \namedref{Theorem}{thm:C} applies.

For $X=\mathbb A_k^n$, one recovers the Weyl algebra
$\mathcal D_k(k[x_1,\ldots,x_n])=A_n(k)$ and its quotient division
ring $D_n(k)$. Thus $D_n(k)$ is Calabi--Yau of dimension
$2n$.

In positive characteristic, $D_n(k)$ is finite-dimensional over its
center, so it is index-symmetric by
\namedref{Corollary}{cor:centrally-finite-indices}.
\end{example}

\subsection{Ore extensions of twisted Calabi--Yau algebras}

\begin{example}
\label{prop:function-field-Ore}
Let $A$ be a $k$-algebra, $\sigma\in\operatorname{Aut}_k(A)$ and
$\delta:A\to A$ a $k$-linear $\sigma$-derivation, i.e.
$$
 \delta(ab)=\sigma(a)\delta(b)+\delta(a)b
 \qquad\text{for } a,b\in A.
$$
The Ore extension $A[t;\sigma,\delta]$ is the $k$-algebra generated by
$A$ and $t$ subject to the relations
$$
 ta=\sigma(a)t+\delta(a)
 \qquad\text{for } a\in A.
$$ If $A$ is a noetherian domain, then
so is $A[t;\sigma,\delta]$.

If $A$ is twisted Calabi--Yau of dimension $d$, then $A[t;\sigma,\delta]$
is twisted Calabi--Yau of dimension $d+1$ by \cite[Theorem~3.3]{LWW}. If
moreover $A$ is a noetherian domain, then $Q(A[t;\sigma,\delta])$ is
twisted Calabi--Yau of dimension $d+1$ by
\namedref{Lemma}{lem:examples-localization}.

In particular, suppose $k$ has characteristic zero and let $K/k$ be a
finitely generated field extension of transcendence degree $n$. Then
$K$ is Calabi--Yau of dimension $n$. Indeed, choose a smooth affine
$k$-domain $R$ with fraction field $K$ such that
$\bigwedge_R^n\operatorname{Der}_k(R)$ is free of rank one. By
\cite[Theorems~3.9 and~3.11]{Kraehmer}, $R$ is homologically smooth,
$T_i(R)=0$ for $i\ne n$, and
$T_n(R)\cong\bigwedge_R^n\operatorname{Der}_k(R)\cong R$ in ${}_R\mathrm{Mod}_R$, so
\namedref{Lemma}{lem:examples-localization} applies. Taking $A=K$, we
conclude that $Q(K[t;\sigma,\delta])$ is twisted Calabi--Yau of
dimension $n+1$.
\end{example}

\subsection{Quantum skew fields}

\begin{example}
\label{ex:examples-quantum}
Let $s\geq1$ and let
$\mathbf q=(q_{ij})\in(k^\times)^{s\times s}$ satisfy
$q_{ii}=1$ and $q_{ij}q_{ji}=1$. Put
$$
 B_{\mathbf q}
 =k\langle X_1,\ldots,X_s\mid X_iX_j=q_{ij}X_jX_i\rangle,
 \qquad
 D_{\mathbf q}=Q(B_{\mathbf q}).
$$
The algebra $B_{\mathbf q}$ is an iterated Ore extension
$$
 B_{\mathbf q}=k[X_1][X_2;\sigma_2]\cdots[X_s;\sigma_s],
 \qquad
 \sigma_j(X_i)=q_{ji}X_i\quad(i<j).
$$
Since $k$ is Calabi--Yau of dimension $0$, applying
\namedref{Example}{prop:function-field-Ore} $s$ times shows that $B_{\mathbf q}$
is a noetherian domain which is twisted Calabi--Yau of dimension $s$
(see also \cite[Proposition~4.1]{LWW}). By
\namedref{Lemma}{lem:examples-localization}, $D_{\mathbf q}$ is twisted
Calabi--Yau of dimension $s$.
\end{example}

\Needspace{13\baselineskip}
\subsection{AS-regular Hopf algebras}
\label{sec:examples-Hopf}

The following common criterion treats enveloping algebras, quantized
enveloping algebras, quantized coordinate rings and group algebras.

\begin{proposition}
\label{prop:examples-Hopf}
Let $H$ be a noetherian AS-regular Hopf $k$-algebra of dimension $d$
which is a domain. Then $H$ and $Q(H)$ are twisted
Calabi--Yau of dimension $d$.
\end{proposition}

\begin{proof}
By definition of AS-regular \cite[Definition~1.2]{BrownZhang}, $H$ is
AS-Gorenstein of finite global dimension $d$. The antipode of $H$ is
bijective by \cite[Corollary~0.3]{LOWY}. Thus $H$ is homologically
smooth by \cite[Lemma~5.2(c)]{BrownZhang}, and $H$ is rigid
Gorenstein by \cite[Proposition~4.5(a)]{BrownZhang}. The latter means
\cite[Definition~4.4]{BrownZhang} that
$$
 \Ext^i_{H^e}(H,H^e)\cong
 \begin{cases}
 0,&i\ne d,\\
 H_\nu,&i=d,
 \end{cases}
$$
for some $k$-algebra automorphism $\nu$ of $H$. Consequently, $H$ is
twisted Calabi--Yau of dimension $d$. As $H$ is a noetherian domain, it
is an Ore domain, and \namedref{Lemma}{lem:examples-localization}
applies.
\end{proof}

\begin{example}[Enveloping algebras]
\label{ex:examples-Lie}
Let $\mathfrak g$ be a finite-dimensional Lie $k$-algebra and put
$H=U(\mathfrak g)$. By the PBW theorem, $H$ is a noetherian Hopf domain,
and it is AS-regular of dimension $d=\dim_k\mathfrak g$
\cite[Proposition~6.3]{BrownZhang}. Hence
\namedref{Proposition}{prop:examples-Hopf} applies to $Q(U(\mathfrak g))$.
\end{example}

For the two quantized families below, let $q$ be an indeterminate over
$\mathbb C$ and put $k=\mathbb C(q)$. This fixes the generic parameter
setting of \cite[Propositions~6.4 and~6.5]{BrownZhang}; no assertion
about arbitrary specializations of $q$ is needed here.

\begin{example}[Quantized enveloping algebras]
\label{ex:quantized-enveloping}

Let $\mathfrak g$ be a finite-dimensional semisimple complex Lie algebra,
and let $U_q(\mathfrak g)$ be the quantized enveloping algebra over $k$
as in \cite[\S6.4]{BrownZhang}. This is a noetherian Hopf domain
\cite[Chapter~I.6]{BG} and is AS-regular of dimension $\dim_{\mathbb C}\mathfrak g$
\cite[Proposition~6.4]{BrownZhang}. Hence
\namedref{Proposition}{prop:examples-Hopf} applies to $Q(U_q(\mathfrak g))$.
\end{example}

\begin{example}[Quantized coordinate rings]
\label{ex:quantized-coordinate}
Let $G$ be a connected semisimple complex algebraic group and let
$\mathcal O_q(G)$ be its standard quantized coordinate ring over $k$,
as in \cite[\S6.5]{BrownZhang}. It is a noetherian Hopf domain
\cite[Chapter~I.8]{BG} and is AS-regular of dimension $\dim G$
\cite[Proposition~6.5]{BrownZhang}. By
\namedref{Proposition}{prop:examples-Hopf}, $Q(\mathcal O_q(G))$ is
twisted Calabi--Yau of dimension $\dim G$.
\end{example}

\begin{example}[Group algebras]
\label{ex:group-algebras}
Here $k$ may again be an arbitrary field. Let $\Gamma$ be a torsion-free polycyclic-by-finite group of Hirsch
length $h$. Then $k[\Gamma]$ is a domain by \cite[Theorem~1.4]{KLM}.
Since $\Gamma$ is torsion-free, the Hopf algebra $k[\Gamma]$ is
noetherian and AS-regular of dimension $h$ by
\cite[Theorem~6.7]{BrownZhang}. Thus
\namedref{Proposition}{prop:examples-Hopf} shows that $Q(k[\Gamma])$ is
twisted Calabi--Yau of dimension $h$.
\end{example}

\Needspace{23\baselineskip}
\subsection{Sklyanin algebras}

\begin{example}
\label{ex:examples-sklyanin}
Let $k$ be algebraically closed, let $n\in\{3,4\}$ with
$\operatorname{char}k\nmid n$, and let $S=S_n(E,\tau)$ be an elliptic
Sklyanin algebra in the sense of \cite[\S1]{SmithTate}. These are connected
graded noetherian AS-regular domains of dimension $n$
\cite[Theorem~1.1 and Corollary~1.3]{TVdB}. They have a nonzero homogeneous
central element of degree $3$ for $n=3$, and of degree $2$ for $n=4$
\cite[\S1, pp.~20--21]{SmithTate}; in particular $Z(S)\neq k$.

By \namedref{Proposition}{prop:degree-zero} in the appendix, $Q(S)$ and
the degree-zero division ring $Q_{\mathrm{gr}}(S)_0$ are twisted
Calabi--Yau, hence index-symmetric by \namedref{Theorem}{thm:C}.
\end{example}

\subsection{Index symmetry for the examples}

\begin{corollary}
\label{cor:examples-indices}
Let $D$ be one of the following division $k$-algebras:
\begin{enumerate}[label=\textup{(\arabic*)}]
\item $Q(\mathcal D_k(R))$, where $\operatorname{char}k=0$ and
$\operatorname{Spec}R$ is a smooth affine integral $k$-variety of
dimension $n$, with $d=2n$;
\item $D_n(k)$, where $\operatorname{char}k=0$, with $d=2n$;
\item $Q(K[t;\sigma,\delta])$ as in
\namedref{Example}{prop:function-field-Ore}, with
$d=\operatorname{trdeg}_k K+1$;
\item $D_{\mathbf q}$ for an $s$-variable quantum affine space, with $d=s$;
\item $Q(H)$ for a noetherian AS-regular Hopf domain $H$ of
dimension $d$. This includes $Q(U(\mathfrak g))$ for a finite-dimensional
Lie $k$-algebra $\mathfrak g$, with $d=\dim_k\mathfrak g$;
$Q(U_q(\mathfrak g))$ and $Q(\mathcal O_q(G))$ under the hypotheses of
\namedref{Examples}{ex:quantized-enveloping} and~\ref{ex:quantized-coordinate},
with $d=\dim_{\mathbb C}\mathfrak g$ and $d=\dim G$, respectively, and
$Q(k[\Gamma])$ for a torsion-free polycyclic-by-finite group $\Gamma$
of Hirsch length $h$, with $d=h$;
\item $Q(S_n(E,\tau))$ and $Q_{\mathrm{gr}}(S_n(E,\tau))_0$ under the
hypotheses of \namedref{Example}{ex:examples-sklyanin}, with $d=n$
and $d=n-1$, respectively.
\end{enumerate}
Then $D$ is index-symmetric. If $E\subseteq D$ is a division
$k$-subalgebra with one finite index, the remaining conclusions of
\namedref{Theorem}{thm:C} hold with the indicated value of $d$.
\end{corollary}

\begin{proof}
By \namedref{Example}{ex:examples-differential} for (1) and (2),
\namedref{Example}{prop:function-field-Ore} for (3),
\namedref{Example}{ex:examples-quantum} for (4),
\namedref{Proposition}{prop:examples-Hopf} for (5), and
\namedref{Example}{ex:examples-sklyanin} and
\namedref{Proposition}{prop:degree-zero} for (6), $D$ is twisted
Calabi--Yau of dimension $d$. Hence \namedref{Theorem}{thm:C} applies.
\end{proof}

\begin{corollary}
\label{cor:banff}
Let $\operatorname{char}k=0$ and let $E\subseteq D_1(k)$ be a division
$k$-subalgebra. If $E$ is noncommutative, then
$[D_1(k):E]_\ell=[D_1(k):E]_r<\infty$. If $E$ is commutative, both
indices are infinite.
\end{corollary}

\begin{proof}
If $E$ is noncommutative, Schofield's finiteness result
\cite[Corollary~35]{Stratiform} gives finite left and right indices, and
they agree by \namedref{Corollary}{cor:examples-indices}. If $E$ is
commutative, both indices are infinite. Indeed, if
$\dim_E({}_E D_1(k))=n<\infty$, right multiplication gives an
anti-homomorphism $D_1(k)\to\operatorname{End}_E({}_E D_1(k))\cong M_n(E)$,
and the image of $\partial x-x\partial=1$ would be a commutator, of
trace $0$, equal to the identity matrix, of trace $n\ne0$ since
$\operatorname{char}k=0$. The right index is treated symmetrically.
\end{proof}

Taking $k=\mathbb C$ and $E\cong_{\mathbb C}D_1(\mathbb C)$ answers the
question from the 2005 Banff workshop \cite[p.~2]{Banff}.

\subsection{Division rings on Artin's list}
\label{sec:artin-list}

Throughout this subsection $k$ is algebraically closed of characteristic
zero. We use Michael Artin's list \cite{ArtinSurfaces} in the form recalled by
Bell, Ingalls and Ramkumar \cite[\S\S1 and 3]{BIR}:
\begin{enumerate}[label=\textup{(\arabic*)}]
\item division algebras finite-dimensional over their centers, which
are finitely generated fields of transcendence degree two over $k$;
\item $Q_{\mathrm{gr}}(S)_0$, the degree-zero part of the graded quotient
ring of a three-dimensional Sklyanin algebra $S$;
\item $Q(K[t;\sigma])$, where $K$ is the function field of a smooth
projective curve of genus zero or one and $\sigma$ is a $k$-automorphism;
\item $Q(K[t;\delta])$, where $K$ is the function field of a smooth
projective curve and $\delta$ is a nonzero $k$-derivation.
\end{enumerate}
In Types~(2)--(4), the division ring is required to be
infinite-dimensional over its center; centrally finite cases are counted in Type~(1).

\begin{corollary}
\label{thm:D}
Every division ring on Artin's list is index-symmetric. Those of
Types~\textup{(2)}--\textup{(4)} are twisted Calabi--Yau of dimension
two, so \namedref{Theorem}{thm:C} applies to them.
\end{corollary}

\begin{proof}
Let $D$ be on Artin's list. For Type~(1), $D$ is index-symmetric by
\namedref{Corollary}{cor:centrally-finite-indices}. For Types~(3) and~(4), $D$ is twisted Calabi--Yau of dimension two by
\namedref{Example}{prop:function-field-Ore}. For Type~(2),
$D=Q_{\mathrm{gr}}(S)_0$ for a three-dimensional Sklyanin algebra $S$, so
$D$ is twisted Calabi--Yau of dimension two by
\namedref{Example}{ex:examples-sklyanin} and
\namedref{Proposition}{prop:degree-zero}.
\end{proof}

\begin{remark}
\label{rem:type2-fixed-ring}
For Type~(2) there is also an argument via fixed rings. By the
description in \cite[\S3]{BIR}, attributed there, through Artin, to
Van den Bergh, $D\cong F^\iota$, where $F=Q(k(C)[t;\sigma])$ for an
elliptic curve $C$ and an infinite-order translation $\sigma$, and
$\iota$ is the involution induced by negation on $C$ and
$t\mapsto t^{-1}$. Since $F$ is of Type~(3),
\namedref{Corollary}{cor:fixed-rings} with $G=\{1,\iota\}$ shows again that
$D$ is twisted Calabi--Yau of dimension two.
\end{remark}

\begin{remark}
\label{rem:BIR-degree}
Both indices of a $k$-embedding between members of Artin's list are finite
\cite[\S3]{BIR}, so they are equal by \namedref{Corollary}{thm:D}. For
example, let $X$ and $Y$ be smooth irreducible affine curves over $k$
with coordinate rings $R_X$ and $R_Y$, and let $D(X)=Q(\mathcal D_k(R_X))$
and $D(Y)=Q(\mathcal D_k(R_Y))$; these are of Type~(4). Bell, Ingalls and
Ramkumar \cite[\S3]{BIR} define the degree of a $k$-embedding
$\varphi\colon D(X)\hookrightarrow D(Y)$ as the minimum of its left and
right indices. Since $D(Y)$ is index-symmetric by
\namedref{Corollary}{thm:D} and the indices of $\varphi$ are finite, the two
indices agree, so the minimum is unnecessary.
\end{remark}

\appendix
\section{Degree-zero division rings}
\label{sec:degree-zero-appendix}

This appendix supplies the descent argument used in
\namedref{Example}{ex:examples-sklyanin}. All tensor products are over
$k$ unless another base is indicated.

\begin{lemma}
\label{lem:central-rational-descent}
Let $D$ be a division $k$-algebra and let $t$ be a central indeterminate.
If $C=D(t)$ is twisted Calabi--Yau of dimension $d$, then $d\geq1$ and
$D$ is twisted Calabi--Yau of dimension $d-1$.
\end{lemma}

\begin{proof}
Put $\Lambda=C\tensor_{k(t)}C^\op$ and
$h=t\tensor1-1\tensor t^\op\in C^e$. The element $h$ is central and
regular: $C^e$ is free over $k(t)\tensor_k k(t)$, a localization of
$k[x,y]$ in which $h=x-y$. Moreover, $\Lambda=C^e/(h)$.

We first show that $\Lambda$ is faithfully flat as a right $D^e$-module.
For a left $D^e[t]$-module $N$, viewed as a $D[t]$-bimodule on which
the two actions of $t$ agree, there is an isomorphism
$$
 \Lambda\tensor_{D^e[t]}N
 \cong C\tensor_{D[t]}N\tensor_{D[t]}C.
$$
Ore localization is flat on both sides, and $D^e[t]$ is free over
$D^e$, so $\Lambda$ is right $D^e$-flat. For any $D$-bimodule $M$,
the inclusion $C\subseteq D((t))$ makes $M((t))$ a $\Lambda$-module.
The map
$$
 \Lambda\tensor_{D^e}M\longrightarrow M((t)),\qquad
 \lambda\tensor m\longmapsto\lambda m
$$
sends $1\tensor m$ to the constant series $m$, proving faithfulness.
The preceding tensor isomorphism also gives
$\Lambda\tensor_{D^e}D\cong C$.

We next descend homological smoothness. Both $\Lambda$ and $C$ are
of type $\FP_\infty$ over $C^e$: for $\Lambda$ use the two-term
resolution given by multiplication by $h$. Successively taking finite
free $\Lambda$-modules mapping onto $C$ and its syzygies shows that
$C$ is of type $\FP_\infty$ over $\Lambda$. Indeed, at each step the
kernel remains of type $\FP_\infty$ over $C^e$
\cite[VIII, Proposition~4.3]{Brown}, so is finitely generated over
$\Lambda$. Faithfully flat descent
(\namedref{Lemma}{lem:flat-FP}) now gives type $\FP_\infty$ for $D$
over $D^e$. Also, $C^e$ is free as a left $D^e$-module, and the
constant-coefficient map $C\subseteq D((t))\to D$ is a $D$-bimodule
retraction. Restricting a finite projective $C^e$-resolution of $C$
therefore gives $\pd_{D^e}D<\infty$. Truncating a resolution by
finitely generated projectives shows that $D$ is homologically smooth.

Finally, write $U_i=T_i(D)$. Tensoring a finite projective bimodule
resolution of $D$ with the standard two-term resolution of $k[t]$ gives
$$
 T_0(D[t])=0,\qquad
 T_{i+1}(D[t])\cong U_i\tensor_k k[t]\quad(i\geq0).
$$
By \namedref{Lemma}{lem:examples-localization} and the tensor
isomorphism above,
$$
 T_0(C)=0,\qquad
 T_{i+1}(C)\cong\Lambda\tensor_{D^e}U_i.
$$
Thus $d\geq1$, and faithful flatness gives $U_i=0$ for $i\neq d-1$
and $U_{d-1}\neq0$. Let $0\neq u\in U_{d-1}$; then $1\tensor u\neq0$ in
$T_d(C)$, as its image in $U_{d-1}((t))$ is $u$. Since $T_d(C)$ is
one-dimensional over $C$ on the left, every $u'\in U_{d-1}$ satisfies
$1\tensor u'=c(1\tensor u)$ for some $c\in C\subseteq D((t))$. Mapping to
$U_{d-1}((t))$ gives $u'=cu$, and $cu$ lies in $U_{d-1}$ only if $c\in D$. Hence
$U_{d-1}=Du$, likewise $U_{d-1}=uD$, and $U_{d-1}\cong D_\nu$ for an
automorphism $\nu$, as in the proof of \namedref{Theorem}{thm:C}.
\end{proof}

\begin{proposition}
\label{prop:degree-zero}
Let $A$ be a connected graded noetherian AS-regular $k$-domain of
dimension $d$ with $Z(A)\neq k$. Then $Q(A)$ and
$Q_{\mathrm{gr}}(A)_0$ are twisted Calabi--Yau of dimensions $d$ and
$d-1$, respectively.
\end{proposition}

\begin{proof}
By \cite[Propositions~4.4(c) and~4.5]{YZ}, $A$ has a finite resolution by
finitely generated graded projective $A^e$-modules, and in the graded sense
$T_i(A)=0$ for $i\neq d$, while $T_d(A)$ is a shift of a twist of $A$ by an
automorphism. Since the ungraded $\Hom$ from a finitely generated graded
module is the sum of its homogeneous components, the same holds without the
grading, with the shift removed.
Thus $A$, and by \namedref{Lemma}{lem:examples-localization} also $L=Q(A)$, is
twisted Calabi--Yau of dimension $d$.

Put $R=Q_{\mathrm{gr}}(A)$ and $D=R_0$. Choose a nonzero homogeneous
central element $c\in A$ of positive degree; such an element exists
by taking a positive-degree homogeneous component of an element of
$Z(A)\setminus k$. The support of the graded division ring $R$ is
$e\mathbb Z$ for some $e>0$. Choose $0\neq u\in R_e$, so
$R=D[u,u^{-1};\sigma]$, and write $\deg c=me$ with $m>0$.
Distinct powers of $c$ have distinct degrees, so $c$ is a central
indeterminate over $D$. Put $K=D(c)$.

Since $u^mc^{-1}\in D^\times$ and $u$ normalizes $K$, the sum
$\bigoplus_{j=0}^{m-1}Ku^j=\bigoplus_{j=0}^{m-1}u^jK$ is a
subring of $L$ containing $R$. Directness follows by clearing
denominators in $D[c,c^{-1}]$ and comparing degrees modulo $me$.
It is a domain finite-dimensional on both sides over $K$, hence a
division ring, and therefore equals $L$. Projection onto the
coefficient of $u^0$ is a $K$-bimodule retraction $L\to K$.
By \namedref{Proposition}{prop:retraction}, $K=D(c)$ is twisted
Calabi--Yau of dimension $d$. Now
\namedref{Lemma}{lem:central-rational-descent} gives the conclusion
for $D$.
\end{proof}

\section*{Declaration of AI use}

The central idea of this paper arose from an unexpected suggestion
by GPT-6 Astra~\cite{GPT6Astra} just after its release. 
We asked the model to investigate the problem in
\cite[Remark~4.4(i)]{LT} in the context of positive characteristic methods.
This problem is a special case of the question about the first Weyl
skew field recorded in \cite{Banff}.

Alongside an approach in positive characteristic, the model suggested
studying the problem through Calabi--Yau algebras and proposed a
preliminary less general form of the bimodule duality now stated in
\namedref{Lemma}{lem:mixed}. This change of viewpoint led to the
twisted Calabi--Yau framework and the results in this paper.

We also used GPT-6 Astra, Fable 5.1, and Opus 5.5 to develop and test
generalizations, draft proofs and text, check computations, and search
the literature. We take full responsibility for the results, proofs,
and references in the final manuscript.


\begin{thebibliography}{99}

\bibitem{ArtinGeometric}
E. Artin,
\emph{Geometric Algebra},
Interscience Tracts in Pure and Applied Mathematics, vol.~3,
Interscience Publishers, New York--London, 1957.

\bibitem{ArtinSurfaces}
M. Artin,
\emph{Some problems on three-dimensional graded domains},
in \emph{Representation Theory and Algebraic Geometry}
(Waltham, MA, 1995), London Math. Soc. Lecture Note Ser.
\textbf{238}, Cambridge University Press, Cambridge, 1997, pp.~1--19.
\href{https://doi.org/10.1017/CBO9780511525995.002}{doi:10.1017/CBO9780511525995.002}.

\bibitem{Baki}
Y.~S. Baki,
\emph{Group-graded twisted Calabi--Yau algebras},
J. Pure Appl. Algebra \textbf{229} (2025), no.~1, article 107849.
\href{https://doi.org/10.1016/j.jpaa.2024.107849}{doi:10.1016/j.jpaa.2024.107849}.
\href{https://arxiv.org/abs/2405.09025v2}{arXiv:2405.09025v2}.

\bibitem{Banff}
M. Artin, C. Ingalls, Z. Reichstein, L. Small, and J.~J. Zhang,
\emph{Interactions between noncommutative algebra and algebraic geometry},
report of the Banff International Research Station workshop 05w5035,
September 10--15, 2005, especially p.~2.
\url{https://www.birs.ca/workshops/2005/05w5035/report05w5035.pdf}.

\bibitem{BIR}
J.~P. Bell, C. Ingalls, and R. Ramkumar,
\emph{Embeddings of quotient division algebras of rings of differential
operators},
Israel J. Math. \textbf{219} (2017), no.~1, 411--430.
\href{https://doi.org/10.1007/s11856-017-1485-z}{doi:10.1007/s11856-017-1485-z}.
\href{https://arxiv.org/abs/1411.3574}{arXiv:1411.3574}.

\bibitem{Brown}
K.~S. Brown,
\emph{Cohomology of Groups},
Graduate Texts in Mathematics \textbf{87}, Springer, New York, 1982.
\href{https://doi.org/10.1007/978-1-4684-9327-6}{doi:10.1007/978-1-4684-9327-6}.

\bibitem{BG}
K. A. Brown and K. R. Goodearl,
\emph{Lectures on Algebraic Quantum Groups},
Advanced Courses in Mathematics CRM Barcelona, Birkh\"auser, Basel, 2002.
\href{https://doi.org/10.1007/978-3-0348-8205-7}{doi:10.1007/978-3-0348-8205-7}.

\bibitem{BrownZhang}
K. A. Brown and J. J. Zhang,
\emph{Dualising complexes and twisted Hochschild (co)homology for
noetherian Hopf algebras},
J. Algebra \textbf{320} (2008), no.~5, 1814--1850.
\href{https://doi.org/10.1016/j.jalgebra.2007.03.050}{doi:10.1016/j.jalgebra.2007.03.050}.
\href{https://arxiv.org/abs/math/0603732}{arXiv:math/0603732}.

\bibitem{Cohn}
P.~M. Cohn,
\emph{Quadratic extensions of skew fields},
Proc. London Math. Soc. (3) \textbf{11} (1961), 531--556.
\href{https://doi.org/10.1112/plms/s3-11.1.531}{doi:10.1112/plms/s3-11.1.531}.

\bibitem{Kraehmer}
U. Kr\"ahmer,
\emph{Poincar\'e duality in Hochschild (co)homology},
in: \emph{New Techniques in Hopf Algebras and Graded Ring Theory} (Brussels, 2006),
S. Caenepeel and F. Van Oystaeyen (eds.),
K. Vlaam. Acad. Belgie Wet. Kunsten, Brussels, 2007, pp.~117--125.

\bibitem{KLM}
P. H. Kropholler, P. A. Linnell and J. A. Moody,
\emph{Applications of a new $K$-theoretic theorem to soluble group rings},
Proc. Amer. Math. Soc. \textbf{104} (1988), no.~3, 675--684.
\href{https://doi.org/10.1090/S0002-9939-1988-0964842-0}{doi:10.1090/S0002-9939-1988-0964842-0}.

\bibitem{LLM}
T. Lambre and P. Le Meur,
\emph{Duality for differential operators of Lie--Rinehart algebras},
Pacific J. Math. \textbf{297} (2018), no.~2, 405--454.
\href{https://doi.org/10.2140/pjm.2018.297.405}{doi:10.2140/pjm.2018.297.405}.
\href{https://arxiv.org/abs/1709.03973}{arXiv:1709.03973}.

\bibitem{LT}
N. Lauritzen and J.~F. Thomsen,
\emph{Two properties of endomorphisms of Weyl algebras},
J. Algebra \textbf{479} (2017), 137--158.
\href{https://doi.org/10.1016/j.jalgebra.2016.12.036}{doi:10.1016/j.jalgebra.2016.12.036}.
\href{https://arxiv.org/abs/1508.02499}{arXiv:1508.02499}.

\bibitem{LWW}
L.-Y. Liu, S.-Q. Wang, and Q.-S. Wu,
\emph{Twisted Calabi--Yau property of Ore extensions},
J. Noncommut. Geom. \textbf{8} (2014), no.~2, 587--609.
\href{https://doi.org/10.4171/JNCG/165}{doi:10.4171/JNCG/165}.
\href{https://arxiv.org/abs/1205.0893}{arXiv:1205.0893}.

\bibitem{LOWY}
J. L\"u, S.-Q. Oh, X. Wang, and X. Yu,
\emph{A note on the bijectivity of the antipode of a Hopf algebra and its
applications},
Proc. Amer. Math. Soc. \textbf{146} (2018), no.~11, 4619--4631.
\href{https://doi.org/10.1090/proc/14140}{doi:10.1090/proc/14140}.
\href{https://arxiv.org/abs/1704.05731v2}{arXiv:1704.05731v2}.

\bibitem{GPT6Astra}
OpenAI,
\emph{GPT-6 Astra},
large language model, 2026.
\url{https://developers.openai.com/api/docs/models/gpt-6-astra}.
Accessed 4 September 2026.

\bibitem{Princeton}
Princeton University,
\emph{Problems of mathematics},
Princeton University Bicentennial Conferences, Series~2, Conference~2,
1947. Reprinted in P. Duren (ed.), \emph{A Century of Mathematics in
America}, Part~II, History of Mathematics \textbf{2}, American
Mathematical Society, Providence, RI, 1989, pp.~309--334,
especially p.~312.
\href{https://www.ams.org/publicoutreach/math-history/hmath2-princeton-2.pdf}{AMS reprint}.

\bibitem{Schofield}
A.~H. Schofield,
\emph{Artin's problem for skew field extensions},
Math. Proc. Cambridge Philos. Soc. \textbf{97} (1985), no.~1, 1--6.
\href{https://doi.org/10.1017/S0305004100062538}{doi:10.1017/S0305004100062538}.

\bibitem{SmithTate}
S.~P. Smith and J. Tate,
\emph{The center of the 3-dimensional and 4-dimensional Sklyanin algebras},
K-Theory \textbf{8} (1994), no.~1, 19--63.
\href{https://doi.org/10.1007/BF00962089}{doi:10.1007/BF00962089}.

\bibitem{Stratiform}
A.~H. Schofield,
\emph{Stratiform simple Artinian rings},
Proc. London Math. Soc. (3) \textbf{53} (1986), no.~2, 267--287.
\href{https://doi.org/10.1112/plms/s3-53.2.267}{doi:10.1112/plms/s3-53.2.267}.

\bibitem{TVdB}
J. Tate and M. Van den Bergh,
\emph{Homological properties of Sklyanin algebras},
Invent. Math. \textbf{124} (1996), no.~1--3, 619--647.
\href{https://doi.org/10.1007/s002220050065}{doi:10.1007/s002220050065}.

\bibitem{VdB}
M. Van den Bergh,
\emph{A relation between Hochschild homology and cohomology for
Gorenstein rings},
Proc. Amer. Math. Soc. \textbf{126} (1998), no.~5, 1345--1348.
\href{https://doi.org/10.1090/S0002-9939-98-04210-5}{doi:10.1090/S0002-9939-98-04210-5}.

\bibitem{Yekutieli}
A. Yekutieli,
\emph{The rigid dualizing complex of a universal enveloping algebra},
J. Pure Appl. Algebra \textbf{150} (2000), no.~1, 85--93.
\href{https://doi.org/10.1016/S0022-4049(99)00032-8}{doi:10.1016/S0022-4049(99)00032-8}.
\href{https://arxiv.org/abs/math/9810016}{arXiv:math/9810016}.

\bibitem{YZ}
A. Yekutieli and J.~J. Zhang,
\emph{Homological transcendence degree},
Proc. London Math. Soc. (3) \textbf{93} (2006), no.~1, 105--137.
\href{https://doi.org/10.1017/S0024611505015698}{doi:10.1017/S0024611505015698}.
\href{https://arxiv.org/abs/math/0412013}{arXiv:math/0412013}.

\end{thebibliography}
\end{document}